\documentclass[12pt, reqno]{amsart}
\usepackage{amsmath, amsthm, amscd, amsfonts, amssymb, graphicx, color}
\usepackage[bookmarksnumbered, colorlinks, plainpages]{hyperref}

\hypersetup{colorlinks=true,linkcolor=red, anchorcolor=green,
citecolor=cyan, urlcolor=red, filecolor=magenta, pdftoolbar=true}
\input{mathrsfs.sty}
\newtheorem{theorem}{Theorem}[section]
\newtheorem{lemma}[theorem]{Lemma}

\newtheorem{corollary}[theorem]{Corollary}
\theoremstyle{definition}
\newtheorem{definition}[theorem]{Definition}
\newtheorem{example}[theorem]{Example}

\theoremstyle{remark}
\newtheorem{remark}[theorem]{Remark}
\numberwithin{equation}{section}

\begin{document}

\title[An extension of orthogonality relations]{An extension of orthogonality relations based\\ on norm derivatives}

\author[A. Zamani and M.S. Moslehian]{Ali Zamani \MakeLowercase{and} Mohammad Sal Moslehian}
\address [A. Zamani]{Department of Mathematics,
Farhangian University, Tehran, Iran} \email{zamani.ali85@yahoo.com}
\address [M. S. Moslehian]{Department of Pure Mathematics,
Ferdowsi University of Mashhad, P.O. Box 1159, Mashhad 91775, Iran}
\email{moslehian@um.ac.ir, moslehian@yahoo.com}

%%%%%%%%%%%%%%%%%%%%%%%%%%%%%%%%%%%%%%%%%%%%%%%%%%%%%
\begin{abstract}
We introduce the relation ${\rho}_{\lambda}$-orthogonality in the setting of normed spaces as an extension
of some orthogonality relations based on norm derivatives, and present some of its essential properties.
Among other things, we give a characterization of inner product spaces via the functional ${\rho}_{\lambda}$.
Moreover, we consider a class of linear mappings preserving this new kind of orthogonality.
In particular, we show that a linear mapping preserving ${\rho}_{\lambda}$-orthogonality
has to be a similarity, that is, a scalar multiple of an isometry.
\end{abstract}
\maketitle

\section{Introduction}
In an inner product space $\big(H, \langle \cdot, \cdot\rangle\big)$,
an element $x\in H$ is said to be orthogonal to $y\in H$ (written as $x\perp y$)
if $\langle x, y\rangle = 0$.
In the general setting of normed spaces, numerous notions of orthogonality
have been introduced. Let $(X, \|\cdot\|)$ be a real normed linear space of dimension at
least 2. One of the most important ones is the concept of the Birkhoff--James orthogonality
($B$-orthogonality) that reads as follows:
If $x$ and $ y$ are elements of $X$, then $x$ is orthogonal to $y$ in the Birkhoff--James
sense \cite{B, J}, in short $x\perp_By$, if
\begin{align*}
\|x + \lambda y\| \geq \|x\| \qquad (\lambda\in\mathbb{R}).
\end{align*}
Also, for $x, y\in X$ the isosceles-orthogonality ($I$-orthogonality) relation in $X$ (see \cite{J}) is defined by
\begin{align*}
x \perp_{I} y \Leftrightarrow \|x + y\| = \|x - y\|.
\end{align*}
One of the possible notions of orthogonality is connected with the so-called
norm's derivatives, which are defined by
\begin{align*}
\rho_{-}(x,y):=\|x\|\lim_{t\rightarrow0^{-}}\frac{\|x+ty\|-\|x\|}{t}
\end{align*}
and
\begin{align*}
\rho_{+}(x,y):=\|x\|\lim_{t\rightarrow0^{+}}\frac{\|x+ty\|-\|x\|}{t}.
\end{align*}
Convexity of the norm yields that the above definitions are meaningful.
The following properties, which will be used in the present paper
can be found, for example, in \cite{A.S.T}.
\begin{itemize}
\item[(i)] For all $x, y \in X$, $\rho_{-}(x, y)\leq \rho_{+}(x, y)$
and $|\rho_{\pm}(x,y)| \leq \|x\|\|y\|.$
\item[(ii)] For all $x, y \in X$ and all $\alpha \in \mathbb{R}$, it holds that
\begin{align*}
\rho_{\pm}(\alpha x,y) = \rho_{\pm}(x,\alpha y)=\left\{\begin{array}{ll}
\alpha \rho_{\pm}(x,y), &\alpha \geq 0,\\
\alpha \rho_{\mp}(x,y), &\alpha< 0.\end{array}\right.
\end{align*}
\item[(iii)] For all $x, y \in X$ and all $\alpha \in \mathbb{R}$,
\begin{align*}
\rho_{\pm}(x,\alpha x + y) = \alpha {\|x\|}^2 + \rho_{\pm}(x,y).
\end{align*}
\end{itemize}
Recall that a support functional $F_x$ at a nonzero $x \in X$ is a norm
one functional such that $F_x(x) = \|x\|$.
By the Hahn--Banach theorem, there always exists
at least one such functional for every $x \in X$.
Recall also that $X$ is smooth at the point $x$ in $X$
if there exists a unique support functional at $x$,
and it is called smooth if it is smooth at every $x \in X$.
It is well known that $X$ is smooth at $x$ if and only if
$\rho_{+}(x,y) = \rho_{-}(x,y)$ for all $y\in X$; see \cite{A.S.T}.

It turns out that the smoothness is closely related to the Gateaux differentiability.
Recall that the norm $\|\cdot\|$ is said to be Gateaux differentiable at $x \in X$ if the
limit
\begin{align*}
f_x(y) = \lim_{t\rightarrow0}\frac{\|x+ty\|-\|x\|}{t}
\end{align*}
exists for all $y\in X$.
We call such $f_x$ as the Gateaux differential at $x$ of $\|\cdot\|$.
It is not difficult to verify that $f_x$ is a bounded linear functional on $X$.
When $x$ is a smooth point, it is easy to see that
$\rho_{+}(x,y) = \rho_{-}(x,y) = \|x\|f_x(y)$ for all $y\in X$.
Therefore $X$ is smooth at $x$ if and only if the norm is the Gateaux differentiable at $x$.

The orthogonality relations related to $\rho_{\pm}$ are defined as follows; see \cite{A.S.T, Mil}:
\begin{align*}
x\perp_{\rho_{\pm}}y \Leftrightarrow \rho_{\pm}(x, y) = 0
\end{align*}
and
\begin{align*}
x\perp_{\rho}y \Leftrightarrow \rho(x,y):=\frac{\rho_-(x,y) + \rho_+(x,y)}{2} = 0.
\end{align*}
Also, the notion of $\rho_*$-orthogonality is introduced in \cite{C.L, M.Z.D} as
\begin{align*}
x\perp_{\rho_*}y \Leftrightarrow \rho_*(x,y) := \rho_-(x,y)\rho_+(x,y) = 0.
\end{align*}
Note that $\perp_{\rho_{\pm}}, \perp_{\rho}, \perp_{\rho_*} \subset \perp_B$.
Furthermore, it is obvious that for a real inner product space all the above
relations coincide with the standard orthogonality given by the inner product.
For more information about the norm derivatives and their
properties, interested readers are referred to \cite{A.S.T, C.W.2, C.W.3, Dra, W}.
More recently, further properties of the relation $\perp_{\rho_*}$
are presented in \cite{M.Z.D}.

Now, we introduce an orthogonality relation as an extension of orthogonality
relations based on norm derivatives ${\rho_{\pm}}$.
\begin{definition}
Let $(X, \|\cdot\|)$ be a normed space, and let $\lambda \in [0, 1]$.
The element $x\in X$ is a ${\rho}_{\lambda}$-orthogonal to $y\in X$,
denoted by $x\perp_{{\rho}_{\lambda}}y$, if
\begin{align*}
{\rho}_{\lambda}(x, y): = \lambda\rho_-(x,y) + (1 - \lambda)\rho_+(x,y) = 0.
\end{align*}
\end{definition}
The main aim of the present work is to investigate the
${\rho}_{\lambda}$-orthogonality in a normed space $X$.
In Section 2, we first give basic properties of the functional ${\rho}_{\lambda}$.
In particular, we give a characterization of inner product spaces based on ${\rho}_{\lambda}$.
Moreover, we give some characterizations of smooth
spaces in terms of ${\rho}_{\lambda}$-orthogonality.
In Section 3, we consider a class of linear mappings preserving this kind of
orthogonality. In particular, we show that a linear mapping preserving
${\rho}_{\lambda}$-orthogonality has to be a similarity, that is, a scalar
multiple of an isometry.
%%%%%%%%%%%%%%%%%%%%%%%%%%%%%%%%%%%%%%%%%%%%%%%%
%%%%%%%%%%%%%%%%%%%%%%%%%%%%%%%%%%%%%%%%%%%%%%%%

\section{${\rho}_{\lambda}$-orthogonality and characterization of inner product spaces}
We start this section with some properties of the functional ${\rho}_{\lambda}$.
The following lemma will be used.
\begin{lemma}\cite[Theorem 1]{Mal}\label{L22}
For any nonzero elements $x$ and $y$ in a normed space $(X, \|\cdot\|)$, it is true that
\begin{align*}
\|x + y\| \leq \|x\| + \|y\| -
\left(2 - \left\|\frac{x}{\|x\|} + \frac{y}{\|y\|}\right\|\right)\min\{\|x\|, \|y\|\}.
\end{align*}
\end{lemma}
%%%%%%%%%%%%%%%%%%%%%%%%%%%%%%%%%%%%%%%%%%%%
\begin{theorem}\label{T23}
Let $(X, \|\cdot\|)$ be a normed space, and let $\lambda \in [0, 1]$. Then
\begin{itemize}
\item[(i)] ${\rho}_{\lambda}(tx, y) = {\rho}_{\lambda}(x, ty)
= t{\rho}_{\lambda}(x, y)$ for all $x, y\in X$ and all $t\geq0$.
\item[(ii)] ${\rho}_{\lambda}(tx, y) = {\rho}_{\lambda}(x, ty)
= t{\rho}_{1 - \lambda}(x, y)$ for all $x, y\in X$ and all $t<0$.
\item[(iii)] ${\rho}_{\lambda}(x, tx + y) = t{\|x\|}^2
+ {\rho}_{\lambda}(x, y)$ for all $x, y\in X$ and all $t\in\mathbb{R}$.
\item[(iv)] If $x$ and $y$ are nonzero elements of $X$ such that
$x\perp_{{\rho}_{\lambda}}y$, then $x$ and $y$ are linearly independent.
\item[(v)] $(\|x\| - \|x - y\|)\|x\|\leq {\rho}_{\lambda}(x, y)
\leq(\|x + y\| - \|x\|)\|x\|$ for all $x, y\in X$.
\item[(vi)] $\big|{\rho}_{\lambda}(x, y)\big| \leq \|x\|\|y\|$ for all $x, y\in X$.
\item[(vii)] If $x$ and $y$ are nonzero elements of $X$, then
\begin{align*}
\left(1 - \left\|\frac{x}{\|x\|} - \frac{y}{\|y\|}\right\|\right)\|x\|\|y\|
\leq {\rho}_{\lambda}(x, y) \leq\left(\left\|\frac{x}{\|x\|}
+ \frac{y}{\|y\|}\right\| - 1\right)\|x\|\|y\|.
\end{align*}
\end{itemize}
\end{theorem}
\begin{proof}
The statements (i)--(vi) follow directly from the definition of the functional ${\rho}_{\lambda}$.
To establish (vii) suppose that $x$ and $y$ are nonzero elements of $X$ and that
$0<t<\frac{\|x\|}{\|y\|}$. Applying Lemma \ref{L22} to $x$ and $ty$, we get
\begin{align*}
\left(2 - \left\|\frac{x}{\|x\|} + \frac{ty}{\|ty\|}\right\|\right)
\min\{\|x\|, \|ty\|\} \leq \|x\| + \|ty\| - \|x + ty\|,
\end{align*}
and hence
\begin{align*}
\left(2 - \left\|\frac{x}{\|x\|} + \frac{y}{\|y\|}\right\|\right)t\|y\|
\leq \|x\| + t\|y\| - \|x + ty\|.
\end{align*}
Thus
\begin{align*}
\frac{\|x + ty\| - \|x\|}{t}\leq \left(\left\|\frac{x}{\|x\|}
+ \frac{y}{\|y\|}\right\| - 1\right)\|y\|.
\end{align*}
It follows that
\begin{align}\label{I21}
\rho_{+}(x, y) \leq \left(\left\|\frac{x}{\|x\|}
+ \frac{y}{\|y\|}\right\| - 1\right)\|x\|\,\|y\|.
\end{align}
Putting $-y$ instead of $y$ in \eqref{I21}, we get
\begin{align}\label{I22}
\rho_{-}(x, y)\geq \left(1 - \left\|\frac{x}{\|x\|}
- \frac{y}{\|y\|}\right\|\right)\|x\|\,\|y\|.
\end{align}
Since $\rho_{-}(x, y) \leq \rho_{+}(x, y)$, from (\ref{I21}) and (\ref{I22}), we reach
\begin{align}\label{I23}
\left(1 - \left\|\frac{x}{\|x\|} - \frac{y}{\|y\|}\right\|\right)\|x\|\,\|y\|
\leq \rho_{+}(x, y) \leq \left(\left\|\frac{x}{\|x\|}
+ \frac{y}{\|y\|}\right\| - 1\right)\|x\|\,\|y\|
\end{align}
and
\begin{align}\label{I24}
\left(1 - \left\|\frac{x}{\|x\|} - \frac{y}{\|y\|}\right\|\right)\|x\|\,\|y\|
\leq \rho_{-}(x, y) \leq \left(\left\|\frac{x}{\|x\|}
+ \frac{y}{\|y\|}\right\| - 1\right)\|x\|\,\|y\|.
\end{align}
Now, from (\ref{I23}), (\ref{I24}), and the definition of ${\rho}_{\lambda}$, the proof is completed.
\end{proof}
%%%%%%%%%%%%%%%%%%%%%%%%%%
\begin{remark}
Since $- 1 \leq 1 - \left\|\frac{x}{\|x\|} - \frac{y}{\|y\|}\right\|$
and $\left\|\frac{x}{\|x\|} + \frac{y}{\|y\|}\right\| - 1 \leq 1$,
the inequality (vii) of Theorem \ref{T23} is an improvement
of the known inequality $\big|{\rho}_{\pm}(x, y)\big| \leq \|x\|\|y\|$.
\end{remark}
%%%%%%%%%%%%%%%%%%%%%%%%%%%%%%%%%%%%%%%
We recall the following lemma which gives a characterization of
Birkhoff--James orthogonality.
%%%%%%%%%%%%%%%%%%%%%%%%%%%%%%%%%%%%%%%%
\begin{lemma}\cite[Theorem 50]{Dra}\label{L21}
Let $X$ be a normed space, and let $x, y\in X$.
Then the following conditions are equivalent:
\begin{itemize}
\item[(i)] $x\perp_B y$.
\item[(ii)] $\rho_-(x,y)\leq 0 \leq \rho_+(x,y)$.
\end{itemize}
\end{lemma}
%%%%%%%%%%%%%%%%%%%%%%%%%%%%%%%%%%%%%%%%%
\begin{theorem}\label{th.001}
Let $X$ be a normed space, and let $\lambda \in [0, 1]$. Then $\perp_{{\rho}_{\lambda}} \subseteq \perp_{B}$.
\end{theorem}
\begin{proof}
Let $x, y\in X$ and $x \perp_{{\rho}_{\lambda}}y$.
Thus $\lambda\rho_-(x, y) = (\lambda - 1)\rho_+(x, y)$.
Since $\rho_-(x, y) \leq \rho_+(x, y)$, we get $\rho_-(x, y) \leq 0 \leq \rho_+(x, y)$.
Therefore, by Lemma \ref{L21}, we conclude that $x \perp_{B} y$.
Hence $\perp_{{\rho}_{\lambda}} \subseteq \perp_{B}$.
\end{proof}
%%%%%%%%%%%%%%%%%%%%%%%%%%
To get our next result, we need the following lemma.
%%%%%%%%%%%%%%%%%%%%%%%%%%%%%
\begin{lemma}\cite[Corollary 11]{Dra}\label{L24}
Let $X$ be a normed space and let $x, y\in X$ with $x\neq 0$. Then there exists a number
$t\in\mathbb{R}$ such that $x\perp_{B} tx + y$.
\end{lemma}
%%%%%%%%%%%%%%%%%%%%%%%%%%%%%%%%%%%%%%%%%
\begin{theorem}
Let $(X, \|\cdot\|)$ be a normed space, and let $\lambda \in [0, 1]$. The following conditions are equivalent:
\begin{itemize}
\item[(i)] $\perp_{B}\subseteq \perp_{{\rho}_{\lambda}}$.
\item[(ii)] $\perp_{B} = \perp_{{\rho}_{\lambda}}$.
\item[(iii)] $X$ is smooth.
\end{itemize}
\end{theorem}
\begin{proof}
(i)$\Rightarrow$(ii) This implication follows immediately from Theorem \ref{th.001}.

(ii)$\Rightarrow$(iii) Suppose (ii) holds. If $\lambda = \frac{1}{2}$, then \cite[Proposition 2.2.4]{A.S.T}
implies that $X$ is smooth. Now, let $\lambda \neq \frac{1}{2}$ and $x, y\in X$.
We should show that $\rho_-(x,y)=\rho_+(x,y)$.
We may assume that $x\neq0$, otherwise $\rho_-(x,y)=\rho_+(x,y)$ trivially holds.
By Lemma \ref{L24}, there exists a number $t\in\mathbb{R}$ such that
$x\perp_{B} tx + y$.
From the assumption, we have ${\rho}_{\lambda}(x, tx + y) = 0$. Hence
$t{\|x\|}^2 + {\rho}_{\lambda}(x, y) = 0$, or equivalently,
\begin{align}\label{I25}
t{\|x\|}^2 + \lambda\rho_-(x, y) + (1 - \lambda) \rho_+(x, y) = 0.
\end{align}
We also have $-x\perp_{B} tx + y$, and so ${\rho}_{\lambda}(-x, tx + y) = 0$.
Thus $-t{\|x\|}^2 - {\rho}_{1 - \lambda}(x, y) = 0$, or equivalently,
\begin{align}\label{I26}
-t{\|x\|}^2 - (1 - \lambda)\rho_-(x, y) - \lambda \rho_+(x, y) = 0.
\end{align}
Therefore, by (\ref{I25}) and (\ref{I26}), we have
\begin{align*}
(2\lambda - 1)\rho_-(x, y) + (1 - 2\lambda) \rho_+(x, y)= 0.
\end{align*}
Consequently, $\rho_-(x,y)=\rho_+(x,y)$. Therefore $X$ is smooth.

(iii)$\Rightarrow$(i) Suppose that $X$ is smooth and that
$x, y\in X$ such that $x\perp_{B}y$. It follows from Lemma \ref{L21} that
$\rho_-(x, y) = \rho_+(x, y) = 0$, and this yields that $x\perp_{{\rho}_{\lambda}}y$.
\end{proof}
%%%%%%%%%%%%%%%%%%%%%%%%%%%%%%%%%%%%%%%%%%%%%%%%%%%%
For nonsmooth spaces, the orthogonalities $\perp_{{\rho}_{\lambda}}$ and $\perp_{B}$
may not coincide.
\begin{example}
Consider the real space $X = \mathbb{R}^2$ equipped with the norm
$\|(\alpha, \beta)\| = \max\{|\alpha|, |\beta|\}$.
Let $x = (1, 1)$ and $y = (0, -1)$. Then, for every $\gamma \in \mathbb{R}$, we have
\begin{align*}
\|x + \gamma y\| = \|(1, 1 - \gamma)\| = \max\{1, |1 - \gamma|\}\geq 1 = \|x\|.
\end{align*}
Hence $x \perp_{B} y$. On the other hand, straightforward computations show that
$\rho_-(x, y) = -1$ and $\rho_+(x, y) = 0$.
It follows that
${\rho}_{\lambda}(x, y) =  -\lambda$.
Thus $x \not\perp_{{\rho}_{\lambda}}y$.
\end{example}
%%%%%%%%%%%%%%%%%%%%%%%%%%%%%%%%%%%%%%
The following result is proved in \cite[Theorem 1]{C.W.2} and \cite[Theorem 3.1]{M.Z.D}.
\begin{theorem}
Let $(X, \|\cdot\|)$ be a real normed space. Then the following conditions are equivalent:
\begin{align*}
&(1) \perp_{\rho_-}\subseteq\perp_{\rho_+}.\quad (2) \perp_{\rho_+}\subseteq\perp_{\rho_-}.\quad (3) \perp_{\rho}\subseteq\perp_{\rho_-}.\\
&(4) \perp_{\rho_-}\subseteq\perp_{\rho}.\quad \,\,(5) \perp_{\rho} \subseteq\perp_{\rho_+}.\quad \,\,(6) \perp_{\rho_+}\subseteq\perp_{\rho}.\\
&(7) \perp_{\rho_*}\subseteq\perp_{\rho_-}.\quad (8) \perp_{\rho_*}\subseteq\perp_{\rho_+}.\quad (9) \perp_{\rho_*}\subseteq\perp_{\rho}.\\
&(10) \perp_{\rho}\subseteq\perp_{\rho_*}.\quad (11) \perp_{B}\subseteq\perp_{\rho_*}.\quad(12)\mbox{ $X$ is smooth}.
\end{align*}
\end{theorem}
%%%%%%%%%%%%%%%%%%%%%%%%%%%%%%%%%%%%%%
The relations $\perp_{\rho_-}$, $\perp_{\rho_+}$, $\perp_{\rho}$, and
$\perp_{{\rho}_{\lambda}}$ are generally incomparable. The following
example illustrates this fact.
%%%%%%%%%%%%%%%%%%%%%%%%%%%%%%%%%%%%%%%%%%%%%%%%%%
\begin{example}\label{ex.001}
Consider the real normed space $X = \mathbb{R}^2$ with the norm
$\|(\alpha, \beta)\| = \max\{|\alpha|, |\beta|\}$.

(i) Let $x = (1, 1)$ and $y = (-\frac{1}{2\lambda}, \frac{1}{2(1 - \lambda)})$.
Simple computations show that
\begin{align*}
\rho_-(x, y) = -\frac{1}{2\lambda} \quad \mbox{and}
\quad \rho_+(x, y) = \frac{1}{2(1 - \lambda)}.
\end{align*}
So we get
\begin{align*}
\rho(x, y) = \frac{2\lambda - 1}{4\lambda(1- \lambda)}
\quad \mbox{and} \quad {\rho}_{\lambda}(x, y) = 0.
\end{align*}
Hence $\perp_{{\rho}_{\lambda}}\nsubseteq \perp_{\rho_-}$,
$\perp_{{\rho}_{\lambda}}\nsubseteq \perp_{\rho_+}$,
and $\perp_{{\rho}_{\lambda}}\nsubseteq \perp_{\rho}$.
%%%%%%%%%%%%%%%%%%%%%%%%%%%%%%%%%%%%%%%%%%%%%%%%%

(ii) Let $z = (1, 1)$, $w = (0, 1)$, $u = (0, -1)$, and $v = (1, -1)$.
It is not hard to compute
\begin{align*}
\rho_-(z, w) = 0, \quad \rho_+(z, w) = 1, \quad {\rho}_{\lambda}(z, w) = 1 - \lambda,
\end{align*}
\begin{align*}
\rho_-(z, u) = -1, \quad \rho_+(z, u) = 0, \quad {\rho}_{\lambda}(z, u) = -\lambda,
\end{align*}
and
\begin{align*}
\rho_-(z, v) = -1, \quad \rho_+(z, v) = 1, \quad \rho(z, v) = 0,
\quad {\rho}_{\lambda}(z, v) = 1 - 2\lambda.
\end{align*}
Thus $\perp_{\rho_-}\nsubseteq \perp_{{\rho}_{\lambda}}$,
$\perp_{\rho_+}\nsubseteq \perp_{{\rho}_{\lambda}}$, and
$\perp_{\rho}\nsubseteq \perp_{{\rho}_{\lambda}}$.
\end{example}
%%%%%%%%%%%%%%%%%%%%%%%%%%%%%%%%%%%%%%%%
The following result gives some characterizations of the smooth normed spaces
based on the ${\rho}_{\lambda}$-orthogonality.
\begin{theorem}\label{th.0101}
Let $(X, \|\cdot\|)$ be a normed space, and let $\frac{1}{2} \neq \lambda \in [0, 1]$. The following conditions are equivalent:
\begin{itemize}
\item[(i)] $\perp_{\rho}\subseteq \perp_{{\rho}_{\lambda}}$.
\item[(ii)] $\perp_{{\rho}_{\lambda}}\subseteq\perp_{\rho}$.
\item[(iii)] $\perp_{{\rho}_{\lambda}} = \perp_{\rho}$.
\item[(iv)] $X$ is smooth.
\end{itemize}
\end{theorem}
\begin{proof}
(i)$\Rightarrow$(iv) Let $x, y\in X\setminus\{0\}$. We have
$x\perp_{\rho}\left(-\frac{\rho(x,y)}{\|x\|^2}x + y\right)$.
It follows from (i) that $x\perp_{{\rho}_{\lambda}}\left(-\frac{\rho(x,y)}{\|x\|^2}x + y\right)$.
From Theorem \ref{T23} (iii), we deduce that
\begin{align*}
-\rho(x,y) + {\rho}_{\lambda}(x,y) =
{\rho}_{\lambda}\left(x, -\frac{\rho(x, y)}{\|x\|^2}x + y\right) = 0.
\end{align*}
Thus ${\rho}_{\lambda}(x, y) = \rho(x,y)$. It ensures that
$(2\lambda - 1)\rho_-(x, y) = (2\lambda - 1)\rho_+(x, y)$,
and therefore we get $\rho_-(x,y)=\rho_+(x,y)$.
It follows that $X$ is smooth.

The other implications can be proved similarly.
\end{proof}
%%%%%%%%%%%%%%%%%%%%%%%%%%%%%%%%%%%%%
If we consider $x\perp_{\rho_+}\left(-\frac{\rho_+(x,y)}{\|x\|^2}x + y\right)$ instead of
$x\perp_{\rho}\left(-\frac{\rho(x,y)}{\|x\|^2}x + y\right)$, then, using the same reasoning
as in the proof of Theorem \ref{th.0101}, we get the next result.
%%%%%%%%%%%%%%%%%%%%%%%%%%%%%%%%%%%%%%
\begin{theorem}
Let $(X, \|\cdot\|)$ be a normed space, and let $\lambda \in (0, 1]$. The following conditions are equivalent:
\begin{itemize}
\item[(i)] $\perp_{\rho_+} \subseteq \perp_{{\rho}_{\lambda}}$
\item[(ii)] $\perp_{{\rho}_{\lambda}}\subseteq\perp_{\rho_+}$.
\item[(iii)] $\perp_{{\rho}_{\lambda}} = \perp_{\rho_+}$.
\item[(iv)] $X$ is smooth.
\end{itemize}
\end{theorem}
%%%%%%%%%%%%%%%%%%%%%%%%%%%%%%%%%%%%%%%%%%%%%%%%%%
In the following result we establish another characterizations of smooth spaces.
%%%%%%%%%%%%%%%%%%%%%%%%%%%%%%%%%%%%%%%%%%%%%%%%%
\begin{theorem}
Let $(X, \|\cdot\|)$ be a normed space, and let $\lambda \in [0, 1)$. The following conditions are equivalent:
\begin{itemize}
\item[(i)] $\perp_{\rho_-} \subseteq \perp_{{\rho}_{\lambda}}$
\item[(ii)] $\perp_{{\rho}_{\lambda}}\subseteq\perp_{\rho_-}$.
\item[(iii)] $\perp_{{\rho}_{\lambda}} = \perp_{\rho_-}$.
\item[(iv)] $X$ is smooth.
\end{itemize}
\end{theorem}
\begin{proof}
The proof is similar to the proof of Theorem \ref{th.0101}, so we omit it.
\end{proof}
%%%%%%%%%%%%%%%%%%%%%%%%%%%%%%%%%%%%%%%%%%%%%%%%5
It is easy to see that, in a real inner product space $X$, the equality
\begin{align}\label{I27}
{\|x + y\|}^4 - {\|x - y\|}^4 = 8\Big({\|x\|}^2\langle x, y\rangle
+ {\|y\|}^2\langle y, x\rangle\Big) \qquad (x, y \in X)
\end{align}
holds, which is equivalent to the parallelogram equality
\begin{align*}
{\|x + y\|}^2 + {\|x - y\|}^2 = 2\big({\|x\|}^2 + {\|y\|}^2\big) \qquad (x, y \in X).
\end{align*}
In normed spaces, the equality
\begin{align*}
{\|x + y\|}^4 - {\|x - y\|}^4 = 8\Big({\|x\|}^2{\rho}_{\lambda}(x, y)
+ {\|y\|}^2 {\rho}_{\lambda}(y, x)\Big) \qquad (x, y \in X).
\end{align*}
is a generalization of the equality \eqref{I27}.
In the following result we give a sufficient condition for a normed space to be smooth.
We use some ideas of \cite[Theorem 5]{Mil}.
%%%%%%%%%%%%%%%%%%%%%%%%%%%%%%%%%%
\begin{theorem}
Let $(X, \|\cdot\|)$ be a normed space and $\lambda \in [0, 1]$. Let
\begin{align}\label{I28}
{\|x + y\|}^4 - {\|x - y\|}^4 = 8\Big({\|x\|}^2{\rho}_{\lambda}(x, y)
+ {\|y\|}^2 {\rho}_{\lambda}(y, x)\Big) \qquad (x, y \in X).
\end{align}
Then $X$ is smooth.
\end{theorem}
\begin{proof}
Let $x, y\in X\setminus \{0\}$ and $\lambda \in (0, 1]$. It follows from \eqref{I28} that
\begin{align*}
8\Big({\|x\|}^2{\rho}_{\lambda}(x, y) &+ {\|y\|}^2 {\rho}_{\lambda}(y, x)\Big)
\\& = {\|x + y\|}^4 - {\|x - y\|}^4
\\& = \lim_{t\rightarrow0^{+}}\Big(\big\|(x + \frac{t}{2}y) + y\big\|^4 - \big\|(x + \frac{t}{2}y) - y\big\|^4\Big)
\\& = \lim_{t\rightarrow0^{+}}8\Big(\big\|x + \frac{t}{2}y\big\|^2{\rho}_{\lambda}(x + \frac{t}{2}y, y)
+ {\|y\|}^2 {\rho}_{\lambda}(y, x + \frac{t}{2}y)\Big)
\\& = \lim_{t\rightarrow0^{+}}8\Big(\big\|x + \frac{t}{2}y\big\|^2{\rho}_{\lambda}(x + \frac{t}{2}y, y)
+ {\|y\|}^2 \big(\frac{t}{2}{\|y\|}^2 + {\rho}_{\lambda}(y, x)\big)\Big)
\\& = 8\Big({\|x \|}^2\lim_{t\rightarrow0^{+}}{\rho}_{\lambda}(x + \frac{t}{2}y, y)
+ {\|y\|}^2 {\rho}_{\lambda}(y, x)\Big).
\end{align*}
Therefore
\begin{align}\label{I29}
\lim_{t\rightarrow0^{+}}{\rho}_{\lambda}(x + \frac{t}{2}y, y) = {\rho}_{\lambda}(x, y).
\end{align}
The equalities \eqref{I28} and \eqref{I29} imply that
\begin{align*}
\rho_+(x, y) &= \|x\|\lim_{t\rightarrow0^{+}}\frac{\|x + ty\| - \|x\|}{t}
\\& = \|x\|\lim_{t\rightarrow0^{+}}\frac{8\Big({\|x +
\frac{t}{2}y\|}^2{\rho}_{\lambda}(x + \frac{t}{2}y, \frac{t}{2}y)
+ {\|\frac{t}{2}y\|}^2 {\rho}_{\lambda}(\frac{t}{2}y, x +
\frac{t}{2}y)\Big)}{t(\|x +ty\| + \|x\|)({\|x + ty\|}^2 + {\|x\|}^2)}
\\& = \|x\|\lim_{t\rightarrow0^{+}}\frac{4{\|x +
\frac{t}{2}y\|}^2{\rho}_{\lambda}(x + \frac{t}{2}y, y)
+ \frac{t^3}{2}{\|y\|}^4 + t^2{\|y\|}^2 {\rho}_{\lambda}(y, x)}{(\|x +ty\|
+ \|x\|)({\|x + ty\|}^2 + {\|x\|}^2)}
\\& = \|x\|\frac{4{\|x\|}^2{\rho}_{\lambda}(x, y)}{(2\|x\|)(2{\|x\|}^2)} = {\rho}_{\lambda}(x, y),
\end{align*}
and hence $\rho_+(x, y) = {\rho}_{\lambda}(x, y)$.
Since ${\rho}_{\lambda}(x, y) = \lambda \rho_-(x, y) + (1 - \lambda)\rho_+(x, y)$,
we get $\rho_-(x, y) = \rho_+(x, y)$.
It follows that $X$ is smooth.

Now, let $\lambda = 0$. Then, by (\ref{I28}) we have
\begin{align}\label{I280}
{\|x + y\|}^4 - {\|x - y\|}^4 = 8\Big({\|x\|}^2{\rho}_+(x, y)
+ {\|y\|}^2 {\rho}_+(y, x)\Big) \qquad (x, y \in X).
\end{align}
If we replace $y$ by $-y$ in (\ref{I280}), then we obtain
\begin{align*}
{\|x - y\|}^4 - {\|x + y\|}^4 = 8\Big(-{\|x\|}^2{\rho}_-(x, y)
- {\|y\|}^2 {\rho}_-(y, x)\Big),
\end{align*}
or equivalently,
\begin{align}\label{I281}
{\|x + y\|}^4 - {\|x - y\|}^4 = 8\Big({\|x\|}^2{\rho}_-(x, y)
+ {\|y\|}^2 {\rho}_-(y, x)\Big) \qquad (x, y \in X).
\end{align}
Add (\ref{I280}) and (\ref{I281}) to get
\begin{align}\label{I282}
{\|x + y\|}^4 - {\|x - y\|}^4 = 8\Big({\|x\|}^2{\rho}(x, y)
+ {\|y\|}^2 {\rho}(y, x)\Big) \qquad (x, y \in X).
\end{align}
Now, by (\ref{I282}) and the same reasoning as in the first part, we conclude that $X$ is smooth.
\end{proof}
%%%%%%%%%%%%%%%%%%%%%%%%%%%%%%%%%%%%%%%%%%%%
Recall that a normed space $(X, \|\cdot\|)$ is uniformly convex whenever, for all
$\varepsilon > 0$, there exists a $\xi > 0$ such that if
$\|x\| = \|y\| = 1$ and $\|x - y\|\geq \varepsilon$, then
$\left\|\frac{x + y}{2}\right\| \leq 1 - \xi$; see, for example, \cite{Dra}.
In the following theorem we state a characterization of uniformly convex spaces via ${\rho}_{\lambda}$.
%%%%%%%%%%%%%%%%%%%%%%%%%%%%%%%%%%%%
\begin{theorem}
Let $(X, \|\cdot\|)$ be a normed space, and let $\lambda \in [0, 1]$. Then the following conditions are equivalent:
\begin{itemize}
\item[(i)] $X$ is uniformly convex.
\item[(ii)] For all $\varepsilon > 0$, there exists a number $\delta > 0$ such that if
$\|x\| = \|y\| = 1$ and $\|x - y\|\geq \varepsilon$, then ${\rho}_{\lambda}(x, y)
\leq \frac{1- \delta^2}{1 + \delta^2}$.
\end{itemize}
\end{theorem}
\begin{proof}
(i)$\Rightarrow$(ii) Let $X$ be uniformly convex, and let $\varepsilon > 0$.
There exists a number $\xi > 0$ such that if
$\|x\| = \|y\| = 1$ and $\|x - y\|\geq \varepsilon$, then
$\left\|\frac{x - y}{2}\right\| \leq 1 - \xi$.
Thus, by Theorem \ref{T23}(v), we obtain
\begin{align*}
{\rho}_{\lambda}(x, y) \leq \|x + y\| - 1 \leq 2(1 - \xi) - 1
= \frac{1- \frac{\xi}{1 - \xi}}{1 + \frac{\xi}{1 - \xi}}.
\end{align*}
Put $\delta = \sqrt{\frac{\xi}{1 - \xi}}$. It follows from the above inequality that
${\rho}_{\lambda}(x, y) \leq \frac{1- \delta^2}{1 + \delta^2}$.

(ii)$\Rightarrow$(i) Suppose (ii) holds.
Let $\varepsilon > 0$, and choose a number $\delta > 0$ such that if
$\|u\| = \|v\| = 1$ and $\|u - v\|\geq \frac{\varepsilon}{4}$, then
${\rho}_{\lambda}(u, v) \leq \frac{1- \delta^2}{1 + \delta^2}$.
Put $\xi = \min\{\frac{\varepsilon}{4}, \frac{\delta^2}{1 + \delta^2}\}$.
Now, let $\|x\| = \|y\| = 1$ and $\|x - y\|\geq \varepsilon$.
If $\left\|\frac{x + y}{2}\right\| = 0$, then
$\left\|\frac{x + y}{2}\right\| \leq 1 - \xi$ is evident.
Therefore, let $\left\|\frac{x + y}{2}\right\| > 0$.
So either $(2 - \|x + y\|) \geq 2\xi$ or
$\|x + y\|\left\|\frac{x + y}{\|x + y\|} - x\right\|\geq \varepsilon - 2\xi$.
(Indeed, otherwise we obtain
\begin{align*}
\|x - y\| = \left\|(2 - \|x + y\|)x - \|x + y\|\left(\frac{x + y}{\|x + y\|} - x\right)\right\|
< 2\xi + \varepsilon - 2\xi = \varepsilon,
\end{align*}
contradicting our assumption.)
If $(2 - \|x + y\|) \geq 2\xi$, then we get $\left\|\frac{x + y}{2}\right\| \leq 1 - \xi$.
In addition, if $\|x + y\|\left\|\frac{x + y}{\|x + y\|} - x\right\|\geq \varepsilon - 2\xi$,
then we reach
\begin{align*}
\left\|\frac{x + y}{\|x + y\|} - x\right\|\geq \frac{\varepsilon - 2\xi}{\|x + y\|}
\geq \frac{\varepsilon - 2\xi}{2}\geq \frac{\varepsilon}{4}.
\end{align*}
Since $\|x\| = \left\|\frac{x + y}{\|x + y\|}\right\| = 1$ and
$\left\|\frac{x + y}{\|x + y\|} - x\right\|\geq \frac{\varepsilon}{4}$,
our assumption yields
\begin{align}\label{I210}
{\rho}_{\lambda}\left(\frac{x + y}{\|x + y\|}, x\right)
\leq \frac{1- \delta^2}{1 + \delta^2}.
\end{align}
By Theorem \ref{T23}(v) and (\ref{I210}), we conclude that
\begin{align*}
\left\|\frac{x + y}{2}\right\| &= \frac{1}{2}\Big( 1 + \big(\|x + y\| - \|(x + y) - x\|\big)\Big)
\\& \leq\frac{1}{2}\left( 1 + \frac{1}{\|x + y\|}{\rho}_{\lambda}(x + y, x)\right)
\\& \leq\frac{1}{2}\left( 1 + \frac{1- \delta^2}{1 + \delta^2}\right)
= 1 - \frac{\delta^2}{1 + \delta^2} \leq 1 - \xi.
\end{align*}
Thus $\left\|\frac{x + y}{2}\right\|\leq 1 - \xi$ and the
proof is completed.
\end{proof}
%%%%%%%%%%%%%%%%%%%%%%%%%%%%%%%%%%%%%
We finish this section by applying our definition of the functional ${\rho}_{\lambda}$
to give a new characterization of inner product spaces.
%%%%%%%%%%%%%%%%%%%%%%%%%%%%%%%%%%%%%%%%
\begin{theorem}\label{T26}
Let $(X, \|\cdot\|)$ be a normed space, and let $\lambda \in [0, 1]$. Then the following conditions are equivalent:
\begin{itemize}
\item[(i)] ${\rho}_{\lambda}(x, y) = {\rho}_{\lambda}(y, x)$ for all $x, y \in X$.
\item[(ii)] The norm in $X$ comes from an inner product.
\end{itemize}
\end{theorem}
\begin{proof}
Obviously, (ii)$\Rightarrow$(i).

Suppose (i) holds. This condition implies that
${\rho}_{1 - \lambda}(x, y) = {\rho}_{1 - \lambda}(y, x)$
for all $x, y \in X$. Indeed Theorem \ref{T23}(ii) implies
\begin{align*}
{\rho}_{1 - \lambda}(x, y) = -{\rho}_{\lambda}(-x, y)
= -{\rho}_{\lambda}(y, -x) = {\rho}_{1 - \lambda}(y, x).
\end{align*}
Now, let $P$ be any two dimensional subspace of $X$.
Define a mapping $\langle \cdot, \cdot\rangle:X\times X\rightarrow\mathbb{R}$ by
\begin{align*}
\langle x, y\rangle := \frac{{\rho}_{\lambda}(x, y) + {\rho}_{1 - \lambda}(x, y)}{2},
\qquad (x, y\in X).
\end{align*}
We will show that $\langle \cdot, \cdot\rangle$ is an inner product in $P$.
It is easy to see that the mapping $\langle \cdot, \cdot\rangle$ is non-negative,
symmetric, and homogeneous.
Therefore, it is enough to show the additivity respect to the second variable.
Take $x, y, z\in P$. We consider two cases:

$\mathbf{Case \,1.}$ $x$ and $y$ are linearly dependent.
Thus $y = tx$ for some $t\in\mathbb{R}$ and so
\begin{align*}
\langle x, y + z\rangle& = \langle x, tx + z\rangle
\\&= \frac{{\rho}_{\lambda}(x, tx + z) + {\rho}_{1 - \lambda}(x, tx + z)}{2}
\\ & = \frac{2t{\|x\|}^2 + {\rho}_{\lambda}(x, z) + {\rho}_{1 - \lambda}(x, z)}{2}
\\& = \langle x, tx\rangle + \langle x, z\rangle
= \langle x, y\rangle + \langle x, z\rangle.
\end{align*}
$\mathbf{Case \,2.}$ $x$ and $y$ are linearly independent.
Hence $z = tx + ry$ for some $t, r\in\mathbb{R}$.
We have
\begin{align*}
\langle x, y + z\rangle& = \langle x, tx + (1 + r)y\rangle
\\& = \frac{{\rho}_{\lambda}\big(x, tx + (1 + r)y\big)
+ {\rho}_{1 - \lambda}\big(x, tx + (1 + r)y\big)}{2}
\\ & = \frac{2t{\|x\|}^2 + {\rho}_{\lambda}\big(x, (1 + r)y\big)
+ {\rho}_{1 - \lambda}\big(x, (1 + r)y\big)}{2}
\\ & = \langle x, tx\rangle + \langle x, (1 + r)y\rangle
\\ & = \langle x, tx\rangle + (1 + r)\langle x, y\rangle
\\& = \langle x, y\rangle + \big(\langle x, tx\rangle
+ \langle x, ry\rangle\big) \qquad (\mbox{by case 1})
\\&= \langle x, y \rangle + \langle x, tx + ry\rangle
= \langle x, y \rangle + \langle x, z\rangle.
\end{align*}
Thus $\langle \cdot, \cdot\rangle$ is an inner product in $P$.
So, by \cite[Theorem 1.4.5]{A.S.T}, the norm in $X$ comes from an inner product.
\end{proof}
%%%%%%%%%%%%%%%%%%%%%%%%%%%%%%%%%%%%%%
%%%%%%%%%%%%%%%%%%%%%%%%%%%%%%%%%%%%%%%%%
\section{Linear mappings preserving ${\rho}_{\lambda}$-orthogonality}
A mapping $T:H\rightarrow K$ between two inner product spaces $H$ and $K$ is said to
be orthogonality preserving if $x\perp y$ ensures $Tx\perp Ty$ for every $x,y\in H$.
It is well known that an orthogonality preserving linear mapping between two inner
product spaces is necessarily a similarity, that is, there exists a positive constant $\gamma$
such that $\|Tx\| = \gamma\|x\|$ for all $x\in H$; see \cite{Ch.1, Z.M.F, Z.C.H.K}.

Now, let $X$ and $Y$ be normed spaces, and let
$\diamondsuit \in \{B, I, \rho_-, \rho_+, \rho, \rho_*, {\rho}_{\lambda}\}$.
Let us consider the linear mappings $T:X\rightarrow Y$, which preserve the
$\diamondsuit$-orthogonality in the following sense:
\begin{align*}
x\perp_{\diamondsuit} y\Rightarrow Tx\perp_{\diamondsuit} Ty \qquad(x,y\in X).
\end{align*}
\begin{remark}
Such mappings can be very irregular, far from being continuous
or linear; see \cite{Ch.1}. Therefore we restrict ourselves to linear mappings only.
\end{remark}
It is proved by Koldobsky \cite{K} (for real spaces) and Blanco and Turn\v{s}ek
\cite{B.T} (for real and complex ones) that a linear mapping $T\colon X\to Y$ preserving
$B$-orthogonality has to be a similarity.
Martini and Wu \cite{M.W} proved
the same result for mappings preserving $I$-orthogonality.
In \cite{C.W.2,C.W.3, W}, for $\diamondsuit \in \{\rho_-, \rho_+, \rho\}$,
Chmieli\'{n}ski and W\'{o}jcik proved that a linear mapping, which preserves
$\diamondsuit$-orthogonality, is a similarity.

Recently, the authors of the paper \cite{M.Z.D} studied $\rho_*$-orthogonality
preserving mappings between real normed spaces. In particular, they showed that
every linear mapping that preserves $\rho_*$-orthogonality is necessarily a similarity
(The same result is obtained in \cite{C.L} by using a different approach for real and complex spaces).

In this section, we show that every ${\rho}_{\lambda}$-orthogonality preserving
linear mapping is necessarily a similarity as well.
%%%%%%%%%%%%%%%%%%%%%%%%%%%%%%%%%%%%%%%%%%
Throughout, we denote by $\mu^n$ the Lebesgue measure on $\mathbb{R}^n$.
When $n = 1$ we simply write $\mu$.
%%%%%%%%%%%%%%%%%%%%%%%%%%%%%
\begin{lemma}\cite[Theorem 1.18]{Ph}\label{L41}
Every norm on $\mathbb{R}^n$ is Gateaux differentiable
$\mu^n$--a.e. on $\mathbb{R}^n$.
\end{lemma}
%%%%%%%%%%%%%%%%%%%%%%%%%%%%%%%%%%%%%%%%%
The following lemma plays a crucial role in
the proof of the next theorem.
\begin{lemma}\cite[Lemma 2.4]{B.T}\label{L42}
Let $\|\cdot\|$ be any norm on $\mathbb{R}^2$, and let $D\subseteq \mathbb{R}^2$
be a set of all nonsmooth points. Then there exists a path
$\gamma : [0, 2] \rightarrow \mathbb{R}^2$ of the form:
\begin{align*}
\gamma(t):= \Bigg \{\begin{array}{ll}
(1, t\xi), & t\in[0, 1],
\\\\
\big(1, (2 - t)\xi + (t - 1)\big), & t\in[1, 2],
\end{array}
\end{align*}
for some $\xi \in \mathbb{R}$, so that $\mu\{t: \gamma(t) \in D\} = 0$.
\end{lemma}
%%%%%%%%%%%%%%%%%%%%%%%%%%%%%%%%%%%%%%%%%%%
We are now in the position to establish the main result of this section.
%%%%%%%%%%%%%%%%%%%%%%%%%%%%%%%%%%%%%%%%%%%
\begin{theorem}\label{T41}
Let $X$ and $ Y$ be normed spaces, and let $T\,:X\longrightarrow Y$ be a nonzero linear bounded mapping.
Then the following conditions are equivalent:
\begin{itemize}
\item[(i)] $T$ preserves ${\rho}_{\lambda}$-orthogonality.
\item[(ii)] $\|Tx\| = \|T\|\,\|x\|$ for all $x\in X$.
\item[(iii)] ${\rho}_{\lambda}(Tx, Ty) = \|T\|^2\,{\rho}_{\lambda}(x, y)$ for all $x, y\in X$.
\end{itemize}
\end{theorem}
\begin{proof}
The implications (ii)$\Rightarrow$(iii) and (iii)$\Rightarrow$(i) are clear
and it remains to prove (i)$\Rightarrow$(ii).
Now we adopt some techniques used by Blanco and Turn\v{s}ek
\cite[Theorem 3.1]{B.T}. Suppose that (i) holds. Clearly we can assume $T \neq 0$.
Let us first show that $T$ is injective.
Suppose on the contrary that $Tx = 0$ for some $x\in X \setminus \{0\}$.
Let $y$ be a element in $X$ which is independent of $x$.
Then we can choose a number $n \in \mathbb{N}$ such that $\frac{\|y\|}{n\|x + \frac{1}{n}y\|} < 1$.
Put $z = x + \frac{1}{n}y$. Therefore Theorem \ref{T23}(vi) implies that
\begin{align}\label{I41}
0 < 1 - \frac{\|y\|}{n\|z\|} = 1 - \frac{\|z\|\,\|y\|}{n{\|z\|}^2}
\leq 1 - \frac{{\rho}_{\lambda}(z, y)}{n{\|z\|}^2}.
\end{align}
On the other hand,
${\rho}_{\lambda}(z, -\frac{{\rho}_{\lambda}(z, y)}{{\|z\|}^2}z + y)
= -\frac{{\rho}_{\lambda}(z, y)}{{\|z\|}^2}{\|z\|}^2 + {\rho}_{\lambda}(z, y) = 0$.
Since $T$ preserves ${\rho}_{\lambda}$-orthogonality, it follows that
\begin{align}\label{I42}
\frac{1}{n}\left(1 -\frac{{\rho}_{\lambda}(z, y)}{{n\|z\|}^2}\right){\|Ty\|}^2
= {\rho}_{\lambda}(Tz, -\frac{{\rho}_{\lambda}(z, y)}{{\|z\|}^2}Tz + Ty) = 0.
\end{align}
Relations (\ref{I41}) and (\ref{I42}) yield $Ty = 0$.
Hence $T = 0$, a contradiction.
We show next that
\begin{align*}
\|x\| = \|y\| \,\Rightarrow \,\|Tx\| = \|Ty\| \qquad (x, y \in X),
\end{align*}
which gives (ii). If $x$ and $y$ are linearly dependent, then $x = ty$ for some $t\in \mathbb{R}$ with $|t| = 1$.
Thus $\|Tx\| = \|tTy\| = \|Ty\|$.
Now let us suppose that $x$ and $y$ are linearly independent.
Let $M$ be the linear subspace spanned by $x$ and $y$.
For $u \in M$,
define ${\|u\|}_T : = \|Tu\|$. Since $T$ is injective, ${\|\cdot\|}_T$ is a norm on $M$.
Let $\Delta$ be the set of all those points $u \in M$ at which at least one of the norms,
$\|\cdot\|$ or ${\|\cdot\|}_T$, is not Gateaux differentiable.
For $u \in M\setminus \Delta$, let $F_u$ and $G_u$ denote the support functionals
at $u$ of $\|\cdot\|$ and ${\|\cdot\|}_T$ on $M$, respectively.
Let $v \in \ker F_u$. Since $(M, \|\cdot\|)$ is smooth at $u$, we obtain
${\rho}_{\lambda}(u, v) = 0$,
and hence ${\rho}_{\lambda}(Tu, Tv) = 0$.
Moreover, since $(M, {\|\cdot\|}_T)$ is smooth at $u$, we have
\begin{align}\label{I44}
{\rho}_{\lambda}(Tu, Tv) = \lambda {\|u\|}_TG_u(v) + (1 - \lambda){\|u\|}_TG_u(v) = \|Tu\|G_u(v),
\end{align}
whence $G_u(v) = 0$.
So, we have $\ker F_u \subseteq \ker G_u$ for all $u \in M\setminus \Delta$, or equivalently there exists a function
$\varphi : M\setminus \Delta \rightarrow \mathbb{R}$ such that $G_u = \varphi(u) F_u$ for all $u \in M\setminus \Delta$.
By \eqref{I44} we get $\|Tu\| = \varphi(u)\|u\|$ for all $u\in M\setminus \Delta$.
So, we conclude that $g_u = \varphi(u) f_u$, where $f_u$ and $g_u$ are the Gateaux
differentials at $u$ of $\|\cdot\|$ and ${\|\cdot\|}_T$, respectively.
Define $L : \mathbb{R}^2 \rightarrow M$ by $L(r, t) := rx + t(y - x)$.
Clearly, $L$ is a linear isomorphism.
Set $D = L^{-1}(M)$. Then $D$ is the set of those points $(r, t) \in \mathbb{R}^2$
at which at least one of the functions $(r, t)\mapsto \|L((r, t)\|$ or
$(r, t)\mapsto {\|L((r, t)\|}_T$ is not Gateaux differentiable.
Both these functions are norms on $\mathbb{R}^4$. Hence, by Lemma \ref{L41}, $\mu^4(D) = 0$.
Let $\gamma : [0, 2] \rightarrow \mathbb{R}^2$ be the path obtained in Lemma \ref{L42}.
Then $\Phi : [0, 2] \rightarrow M$ defined by
\begin{align*}
\Phi(t) := \frac{\|x\|}{\|L(\gamma(t))\|} L(\gamma(t)) \qquad (t\in [0, 2]),
\end{align*}
is a path from $x$ to $y$ such that $\|\Phi(t)\| = \|x\|$ and
$\mu\{t:\,\Phi(t)\in \Delta\} = \mu\{t:\,\gamma(t)\in D\} = 0$.
Note that $t \mapsto \|L(\gamma(t))\|$ and $t \mapsto {\|L(\gamma(t))\|}_T$
are Lipschitz functions and, therefore, are absolutely continuous. Indeed,
if $t_1, t_2\in[0, 1]$, then
\begin{align*}
\Big|\|L(\gamma(t_1))\| - \|L(\gamma(t_2))\|\Big| \leq |\xi||t_1 - t_2|\|y - x\|.
\end{align*}
In addition, if $t_1, t_2\in[1, 2]$, then
$\Big|\|L(\gamma(t_1))\| - \|L(\gamma(t_2))\|\Big|
\leq |1 - \xi||t_1 - t_2|\|y - x\|$.
Finally, if $t_1\in[0, 1]$ and $t_2\in[1, 2]$, then
\begin{align*}
\Big|\|L(\gamma(t_1))\| - \|L(\gamma(t_2))\|\Big|
\leq (1 + |\xi|)|t_1 - t_2|\,\|y - x\|.
\end{align*}
So $t \mapsto \|L(\gamma(t))\|$ satisfies Lipschitz conditions.
Similarly, $t \mapsto {\|L(\gamma(t))\|}_T$ satisfies
Lipschitz conditions. It follows that
${\|\Phi(t)\|}_T = \frac{\|x\|{\|L(\gamma(t))\|}_T}{\|L(\gamma(t))\|}$
is absolutely continuous and that
$\mu\big\{t:\, \Phi'(t)\,\,\mbox{does not exist}\big\}
= \mu\big\{t:\,\, {\|L(\gamma(t))\|}'\,\mbox{does not exist}\big\} = 0$.
Since $t \mapsto \|\Phi(t)\| = \|x\|$ is a constant function, we obtain
${{\|\Phi(t)\|}'}_T = 0$ $\mu$--a.e. on $[0, 2]$. Thus $t \mapsto {\|\Phi(t)\|}_T$ is a constant
function, and we arrive at $\|Tx\| = \|Ty\|$.
\end{proof}
%%%%%%%%%%%%%%%%%%%%%%%%%%%%%%%%%%%%%%%%%%%
Finally, taking $X = Y$ and $T = id$, one obtains, from Theorem \ref{T41},
the following result.
%%%%%%%%%%%%%%%%%%%%%%%%%%%%%%%%%%%%%%%%%%%
\begin{corollary}
Let $X$ be a normed space endowed with two norms ${\|\cdot\|}_1$ and ${\|\cdot\|}_2$,
which generate respective functionals ${\rho}_{\lambda, 1}$ and ${\rho}_{\lambda, 2}$.
Then the following conditions are equivalent:
\begin{itemize}
\item[(i)] There exist constants $0 < m \leq M$ such that
\begin{align*}
m|{\rho}_{\lambda, 1}(x, y)| \leq |{\rho}_{\lambda, 2}(x, y)|
\leq M |{\rho}_{\lambda, 1}(x, y)| \qquad (x, y\in X).
\end{align*}
\item[(ii)] The spaces $(X, {\|\cdot\|}_1)$ and $(X, {\|\cdot\|}_2)$
are isometrically isomorphic.
\end{itemize}
\end{corollary}
%%%%%%%%%%%%%%%%%%%%%%%%%%%%%%%
\textbf{Acknowledgement.}
This research is supported by a grant from the
Iran National Science Foundation (INSF- No. 95013683).
%%%%%%%%%%%%%%%%%%%%%%%%%%%%%%%%
\bibliographystyle{amsplain}

\end{document}